\documentclass[12pt]{amsart}
\usepackage{etex}
\usepackage[margin=1in]{geometry}

\usepackage{amsmath, amssymb, amsfonts, amsthm}
\usepackage{mathtools}
\usepackage{thmtools,thm-restate}

\usepackage{parskip}

\usepackage{etoolbox}
\AtEndEnvironment{thm}{\vspace{-0.5\baselineskip}}
\AtEndEnvironment{prop}{\vspace{-0.5\baselineskip}}
\AtEndEnvironment{lem}{\vspace{-0.5\baselineskip}}
\AtEndEnvironment{cor}{\vspace{-0.5\baselineskip}}
\AtEndEnvironment{rem}{\vspace{-0.5\baselineskip}}

\theoremstyle{plain}
\newtheorem{thm}{Theorem}[section]
\newtheorem{prop}[thm]{Proposition}
\newtheorem{cor}[thm]{Corollary}

\newtheorem{lem}[thm]{Lemma}

\newtheorem*{main_thm*}{Main Theorem}

\theoremstyle{definition}
\newtheorem{defn}[thm]{Definition}

\newtheorem{rem}[thm]{Remark}

\usepackage[shortlabels]{enumitem}
\setlist[description]{
	font = \normalfont}
\setlist[enumerate]{label=(\arabic*)}

\usepackage[british]{isodate}

\usepackage{bm}

\usepackage{hyperref}

\newcommand{\rrank}[2][1.5]{%
	{}\mkern#1mu\overline{\mkern-#1mu#2}}

\DeclareMathOperator{\rank}{rank}
\DeclareMathOperator{\kr}{\rrank{\kappa}}
\newcommand{\bs}{\backslash}

\DeclareMathOperator{\G}{\mathcal{G}}

\usepackage{relsize}
\DeclareMathOperator*{\FreeProd}{\mathlarger{\mathlarger\ast}}

\title{A Bass--Serre theoretic proof of a theorem of Burns and Romanovskii}
\author{Naomi Andrew}
\date{}
\begin{document}

\begin{abstract}
A well known theorem of Burns and Romanovskii states that a free product of subgroup separable groups is itself subgroup separable. We provide a proof using the language of immersions and coverings of graphs of groups, due to Bass.
\end{abstract}

\maketitle

\section{Introduction}

Subgroup separability is a strengthening of residual finiteness. It has many equivalent definitions; we will use the following:

\begin{defn}
Let $G$ be a group and $H$ be a subgroup of $G$. Say $H$ is \emph{separable in $G$} if for every element $g \in G\smallsetminus H$, there is a finite index subgroup $K$ of $G$ containing $H$ and not $g$.

If every finitely generated subgroup of $G$ is separable in $G$, say that $G$ is \emph{subgroup separable}.
\end{defn}

This note is concerned with the following theorem, giving that subgroup separability is closed under finite free products:

\begin{main_thm*}
Suppose $G$ is a finite free product of subgroup separable groups. Then $G$ is itself subgroup separable.
\end{main_thm*}

This theorem is originally due (independently) to Romanovski \cite{Romanovskii1969} and Burns \cite{Burns1971}, and there is a subsequent proof due to Wilton \cite{Wilton2007}.

The object of this paper is to provide a new proof of this theorem, generalising the proof Stallings gives in~\cite{Stallings1983} that free groups are subgroup separable (a theorem originally due to Hall \cite{HallFreeGpLERF}; see also \cite{BurnsFreeGpLERF}) to graphs of groups with trivial edge groups.

The notions of immersions and coverings of graphs of groups, due to Bass, are rather more technical than those used by Stallings for graphs. So we begin by covering the necessary definitions for graphs of groups, then the notion of \emph{Kurosh rank} for a subgroup of a free product (given an action on a tree). Given a group acting on any set (or in particular a tree) we provide a way to calculate the index of a subgroup from its action in Lemma~\ref{lem:counting_orbits}. Finally in Section 4 we combine these results to show how to complete an immersion of graphs of groups to a cover, and how this implies the Burns--Romanovskii theorem.

\section{Graphs of Groups}
We follow Bass' exposition \cite{Bass1993}, although we change some notation. There are other sources covering the same material, such as~\cite{Serre2003}. Unlike many expositions, we put the action on the right.

Graphs of groups are a combinatorial tool encoding group actions on trees: they consist of a graph corresponding to the quotient together with edge and vertex groups corresponding to stabilisers.

\begin{defn}
\label{def:graph}
A graph $\Gamma$ consists of a set of vertices $V\Gamma$ and a set of edges $E\Gamma$, together with two maps: $\iota: E\Gamma \to V\Gamma$; and an involution $E\Gamma \to E\Gamma, e \to \overline{e}$. We also define $\tau: E\Gamma \to V\Gamma, \tau(e)=\iota(\overline{e})$. An \emph{orientation} of $\Gamma$ is a choice of one edge from each pair $\{e,\overline{e}\}$.
\end{defn}

\begin{defn}
\label{def:graph_of_groups}
A \emph{Graph of Groups}, $\mathcal{G}$, consists of \begin{itemize}
\item a connected graph $\Gamma_\mathcal{G}$;
\item for each vertex $v$ of $\Gamma_\mathcal{G}$, a group $G_v$;
\item for each edge $e$ of $\Gamma_\mathcal{G}$, a group $G_e$ such that $G_e = G_{\overline{e}}$ and there is a monomorphism $\alpha_e:G_e \to G_{\tau(e)}$.
\end{itemize}
\end{defn}

Where the graph of groups is clear, we may just refer to $\Gamma$ for the underlying graph.

There are two main ways of defining the fundamental group and universal cover of a graph of groups: by a maximal tree, and by considering loops at a base point. We follow Bass, and consider paths and loops in the graph of groups.

\begin{defn}[Paths]
Let $F(\mathcal{G})$ be the group generated by all the vertex groups and all the edges of $\mathcal{G}$, subject to relations $e\alpha_e(g)\overline{e}=\alpha_{\overline{e}}(g)$ for $g \in G_e$. Note that taking $g=1$ this gives that $e^{-1}=\overline{e}$, as expected.

Define a \emph{path} (of length $n$) in $F(\mathcal{G})$ to be a sequence $g_0e_1g_1 \dots e_ng_n$, where each $e_i$ has $\iota(e_i)=v_{i-1}$ and $\tau(e_i)=v_i$ for some vertices $v_i$ (so there is a path in the graph), and each $g_i \in G_{v_i}$. A \emph{loop} is a path where $v_0=v_n$.
\end{defn}

The set of all paths in $F(\mathcal{G})$ forms a groupoid (sometimes called the fundamental groupoid of $\mathcal{G}$).

\begin{defn}[Reduced paths]
A path is \emph{reduced} if it contains no subpath of the form  $e\alpha_e(g)\overline{e}$ (for $g \in G_e$). A loop is \emph{cyclically reduced} if, in addition to being reduced, $e_n(g_ng_0)e_1$ is not of the form  $e\alpha_e(g)\overline{e}$.
\end{defn}

Every path is equivalent (by the relations for $F(\mathcal{G})$) to a reduced path, and similarly every loop is equivalent to both a reduced loop and a cyclically reduced loop. In general these reduced representations are not unique, although all equivalent (cyclically) reduced paths (or loops) will have the same edge structure. Note that a cyclically reduced loop might not be at the same vertex as the original loop.

\begin{defn}The \emph{fundamental group of $\mathcal{G}$} at a vertex $v$ is the set of loops in $F(\mathcal{G})$ at $v$, and is denoted $\pi_1(\mathcal{G},v)$.
\end{defn}

The isomorphism class of this group does not depend on the vertex chosen. (In fact, the two groups obtained by choosing different base vertices are conjugate in the groupoid.)

We define the Bass--Serre tree (or universal cover) in the corresponding way:

\begin{defn}[Bass--Serre Tree]
Let $T$ be the graph formed as follows: the vertex set consists of `cosets' $G_{w}p$, where $p$ is a path in $F(\mathcal{G})$ from $w$ to $v$. There is an edge(-pair) joining two vertices $G_{w_1}p_1$ and $G_{w_2}p_2$ if $p_1=eg_{w_2}p_2$ or $p_2=eg_{w_1}p_1$ (with $g_w \in G_w$).
\end{defn}

This graph is a tree, and there is a right action of $\pi_1(\mathcal{G},v)$ on the vertex set, since this multiplication is possible in the groupoid, and the paths will still start at $v$. This action preserves adjacency and is without inversions, and so $\pi_1(\mathcal{G},v)$ acts on $T$.

There is another construction, that takes a group action on a tree and returns a graph of groups:

\begin{defn}[Quotient graph of groups]
\label{def:quotient_gog}
Suppose a group $G$ acts on a tree $T$. Form a graph of groups whose underlying graph to be the quotient graph of the action, with edge and vertex groups are assigned as follows: choose subtrees $T^v \subseteq T^e$ such that $T^v$ contains exactly one representative of each vertex orbit (that is, a lift of a maximal tree in the quotient), and $T^e$ exactly one representative of each edge orbit, in such a way that at least one end of every edge is in $T^v$. We abuse notation a little by identifying vertices in $T^v$ and edges in $T^e$ with their orbits (that is, their image in the quotient graph). Set the vertex and edge groups to be the stabilisers $G_v$ and $G_e$. To define the monomorphisms, we choose elements $g_v \in G$ which act to bring each vertex of $T^e$ into $T^v$: if $v \in T^v$ then set $g_v =1$, and otherwise choose any element with this property. Now we may set the monomorphisms $\alpha_e$ to be the composition of the inclusion with conjugation by our chosen elements (so $s \mapsto g_{\tau(e)}^{-1}sg_{\tau(e)}$).
\end{defn}
In many cases, the full complexity of this definition is unnecessary: we can consider the stabiliser of \emph{any} orbit representative and assert that the injection is the composition of an inclusion and the relevant conjugation.

\begin{thm}
\label{thm:structure_theorem}
Up to isomorphism of the structures concerned, the processes of constructing the quotient graph of groups, and of constructing the fundamental group and Bass--Serre tree are mutually inverse.
\end{thm}

From the perspective of groups acting on trees, the isomorphisms required are an isomorphism between the original group and the fundamental group, and an equivariant isometry between the original tree and the Bass--Serre tree. From the perspective of graphs of groups, they are an isomorphism of underlying graphs, together with isomorphisms of corresponding edge and vertex groups (and these must respect the edge monomorphisms).
This is the fundamental result linking actions on trees with splittings of groups. 

\subsection*{Morphisms, immersions and covers}
If $G$ is the fundamental group of a graph of groups, then any subgroup of $G$ will also act on the Bass--Serre tree, and this action will give a quotient graph of groups carrying that subgroup. In the case of a free action (where $G$ must be a free group), then we know that the quotient graph is a cover of the original graph - in fact, there is a correspondence between covers and subgroups. This point of view has been fruitful for investigating free groups, and is the main tool of Stallings' paper \cite{Stallings1983}. The aim of Bass' definitions of morphisms and covers (and immersions) is to recover the same correspondence for graphs of groups.

There is a lot of structure, and so any definition of a morphism must feature a graph map and several group homomorphisms. 
It turns out that slightly more data is needed as well, in the form of group elements attached to each edge and vertex.

The definitions we give here are specialised to the case of free products - that is, when all $G_e$ are trivial. In general the elements $\delta_e$ defined below must satisfy conditions involving the edge group inclusions, but these are automatically satisfied for any choices with trivial edge groups.

\begin{defn}
Suppose $\mathcal{H}$ and $\mathcal{G}$ are graphs of groups with all edge groups trivial.
A \emph{morphism} of graphs of groups $\Phi:\mathcal{H} \to \mathcal{G}$ consists of: \begin{itemize}
\item a graph morphism $\varphi: \Gamma_\mathcal{H} \to \Gamma_\mathcal{G}$;
\item a group homomorphism $\phi_v:H_v \to G_{\varphi(v)}$ for every vertex $v \in \Gamma_\mathcal{H}$;
\item an element $\lambda_v$ in $\pi_1(\mathcal{G},\varphi(v))$ for every vertex $v$ of $\Gamma_\mathcal{H}$;
\item an element $\delta_e \in G_{\varphi(\iota(e))}$ for every edge of $\Gamma_\mathcal{H}$.
\end{itemize}
\end{defn}

Such a morphism induces maps on the structures that can be defined from a graph of groups, as follows: \begin{itemize}
\item A homomorphism (of groups) $F(\mathcal{H})\to F(\mathcal{G})$ by $s \mapsto \lambda_v^{-1}\phi_v(s)\lambda_v$ for $s \in H_v$ and $e \mapsto \lambda_{\iota(e)}^{-1}\delta_e^{-1} e \delta_{\overline{e}}\lambda_{\tau(e)}$.
\item A homomorphism $\Phi_P$ of fundamental groupoids, by restricting that map to the paths in $F(\mathcal{H})$. Note that, in this case, for each edge $e$, the extra elements introduced at $e$, $\iota(e)$ and $\tau(e)$ will cancel to leave $\delta(e)$ and $\delta(\overline{e})$ which are elements of the vertex groups at either end.
\item A homomorphism $\Phi_v: \pi_1(\mathcal{H},v) \to \pi_1(\mathcal{G},\varphi(v))$ of the fundamental groups, by further restricting the above map to loops at $v$.
\item An equivariant graph map $\tilde{\Phi}$ on the Bass--Serre trees, defined on vertices by $H_wp \mapsto G_{\varphi(w)}\lambda_w\Phi_P(p)$.
\end{itemize}

Additionally, we can define a `local map' at each edge of $\mathcal{G}$. (Requiring edge stabilisers to be trivial simplifies this considerably compared to Bass' general definition.)

Given a vertex $v$ and an edge with $\tau(e)=v$, the lifts of $e$ at a single vertex in the Bass--Serre tree correspond to the elements of $G_v$, by identifying the edge $[G_vp,G_wesp]$ with the element $s$.

Given a morphism $\Phi: \mathcal{H} \to \mathcal{G}$, let $v$ be a vertex of $\Gamma_{\mathcal{H}}$ and $f$ be an edge of $\Gamma_{\mathcal{G}}$ with $\tau(f)=\varphi(v)$. Define a map
\[\Phi_{v/f}: \coprod_{e \in \varphi^{-1}(f),\tau(e)=v} H_v \to G_{\varphi(v)}\]
by
\[h \mapsto \delta_{\overline{e}}\phi_v(h).\]

Alternatively, we can view $\Phi_{v/f}$ as a map $H_v \times \{e \in E(\mathcal{H}): \iota(e)= v, \varphi(e)=f\} \to G_{\varphi(v)}$ taking $(H_v,e) \mapsto \delta_e \phi_v(H_v)$. These maps are useful for ``locally'' understanding the image of the Bass--Serre tree under a morphism: see Proposition~\ref{prop:immersion_charac}.

Given two group actions on trees, and an equivariant map between the trees we can induce a graph of groups morphism between the quotient graphs of groups. We continue to assume that the actions are free on edges.

\begin{prop}
\label{prop:induced_morphism}
Suppose $S$ is an $H$-tree, $T$ a $G$-tree, $\psi: H \to G$ is a homomorphism and $f: S \to T$ is a $\psi$-invariant graph map. (That is, $f$ sends vertices to vertices, edges to edges, preserves adjacency, and $vhf=(vf)(h\psi)$.) Let $\mathcal{H}$ and $\mathcal{G}$ be the quotient graphs of groups corresponding to the actions of $H$ on $S$ and $G$ on $T$ respectively. Then $\psi,f$ induce a graph of groups morphism $\mathcal{H} \to \mathcal{G}$, which (after the isomorphisms required by Theorem~\ref{thm:structure_theorem}) recover $\psi$ and $f$ as maps of fundamental groups and Bass--Serre trees.
\end{prop}

For details, and full proofs, see~\cite[Section 4]{Bass1993}. Here we give sufficient details to explain how the induced morphism is constructed.
First, since $f$ was $\psi$-equivariant it induces a graph map $\varphi$ on the quotients $S/H \to T/G$: this is our map between the underlying graphs. Let $S^v$, $S^e$, $T^v$ and $T^e$ be the subtrees of $S$ and $T$ used in Definition~\ref{def:quotient_gog}, and let $h_u$ and $g_v$ be the elements given there which bring vertices of $S^e$ and $T^e$ into $S^v$ and $T^v$ respectively.

For vertices $v$ in $S^v$, choose a $k_v$ in $G$, so that $f(v)k_v$ is in $T^v$; similarly for an edge $e$ let $k_e$ be an element of $G$ with $f(e)k_e$ in $T^e$, and $k_e=k_{\overline{e}}$. Since $f$ is $\psi$-equivariant, $\psi$ takes stabilisers to stabilisers, though not necessarily of the preferred representative of each orbit. Define $\phi_v:H_v \to G_{f(v)k_v}$ by $s \mapsto k_v^{-1}\psi(s)k_v$.

To define a morphism also requires elements $\lambda_v$ and $\delta_e$. For an edge $e$ in $S^e$, let $v=\iota(e)$, $x=f(v)k_e$ and $y=vh_v$. Then let $\delta_e=g_x^{-1}k_e^{-1}\psi(h_v)k_y$. To see that this is indeed an element of $G_{\varphi(v)}$, observe that both $\psi(h_v)k_y$ and $k_eg_x$ act to bring $f(v)$ into $T^v$, and the vertex group $G_{\varphi(v)}$ (of $\mathcal{G}$) is defined as the stabiliser of the vertex of $T^v$ in the same orbit as $f(v)$.

We will want to let $\lambda_v=k_v^{-1}$; however to be in the right group we must first apply the isomorphism (of Theorem~\ref{thm:structure_theorem} from $G \to \pi_1(\mathcal{G},\varphi(v))$. (This can be thought of as ``reading'' the path between $v$ and $vk_v$ in $T$.)

There is usually some choice as to the subtrees used to construct the quotient graph of groups. In particular, if we arrange for $f(S^v)$ to (maximally) intersect $T_v$, we may choose several $\lambda_v$ to be $1$, simplifying the morphism and allowing choices of basepoint (in $\mathcal{H}$) so that the map on fundamental groups is ``as written'' -- meaning it does not involve a conjugation by a non-trivial $\lambda_v$.

We are most interested in studying subgroups $H$ of a group $G$ with an action on a tree $T$, so usually $\psi$ is an inclusion, and $f$ is either the inclusion $T_H \to T$ (sometimes a slightly larger $H$-invariant tree) or the identity $T \to T$. In this case, we should expect the induced morphism to have good properties, since the map on trees makes no identifications. These good properties are characterised by the morphism being a \emph{cover} or \emph{immersion}.

In the context of a graph (with no groups) a covering map corresponds to the usual topological definition, and an immersion relaxes ``locally bijective" to ``locally injective". This allows the universal cover of the immersed graph to be strictly contained in the original universal cover.

The Bass--Serre tree gives the ``universal cover'' in this world: of course it is not a true cover, since an edge may have many (even infinitely many) preimages at each vertex. Similarly, our covers and immersions might have several preimages of an edge at a vertex:

\begin{defn}[{\cite[Definition 2.6]{Bass1993}}]
\label{def:immersion,cover}
A morphism $\Phi: \mathcal{H} \to \mathcal{G}$ is an \emph{immersion} if \begin{enumerate}
\item each $\phi_v : H_v \to G_{\phi(v)}$ is injective and
\item each $\Phi_{v/e}$ is injective
\end{enumerate}
And a \emph{covering} if the second condition is replaced by
\begin{enumerate}[(2')]
\setcounter{enumi}{1}
\item each $\Phi_{v/e}$ is bijective.
\end{enumerate}
\end{defn}

Bass proves that these properties exactly characterise the situation of a subgroup acting on a subtree:

\begin{prop}[{\cite[Proposition 2.7]{Bass1993}}]
\label{prop:immersion_charac}
A morphism $\Phi$ is an immersion if and only if $\Phi_{v_0}$ (on fundamental groups), and $\tilde{\Phi}$ (on Bass--Serre trees) are injective. Furthermore, it is a covering if and only if $\Phi_{v_0}$ (on fundamental groups) is injective and $\tilde{\Phi}$ (on Bass--Serre trees) is bijective.
\end{prop}

Viewing $\Phi_{v/f}$ as a map $H_v \times \{e \in E(\mathcal{H}): \iota(e)= v, \varphi(e)=f\} \to G_{\varphi(v)}$ taking $(H_v,e) \mapsto \delta_e \phi_v(H_v)$, we have that it will be injective if and only if $\phi_v$ is, and the $\delta_e$ represent different right cosets of $G_{\varphi(v)}/\phi_v(H_v)$.

One way to construct immersions is by taking a \emph{subgroup of subgroups}: restrict to a subgraph of the underlying graph, and take subgroups of each vertex group. Then letting $\Phi$ consist of the graph and group inclusions, and all $\lambda_v$ and $\delta_e$ trivial, this is an immersion.

\section{Kurosh rank and finite index subgroups of free products}

Recall Kurosh's theorem about subgroups of free products:

\begin{thm}[\cite{KuroshSubgroup}]
Suppose $G$ is a free product $\FreeProd G_i$ (over some index set $I$) and $H$ is a subgroup of $G$. Then $H \cong (\FreeProd H_j) \ast F$ where each $H_j$ is isomorphic to an intersection $H \cap G_i^{k_i}$ of $H$ with a conjugate of some $G_i$ Further, the set $\{H_j\}$ is unique up to conjugation and reindexing, and the rank of $F$ is uniquely determined.
\end{thm}

The idea of Kurosh rank is inspired by this theorem, and aims to measure the ``complexity'' of such a subgroup in terms of its free factors. Here we take the approach of~\cite{antolin2014kurosh} of defining it with respect to an action on a tree:

\begin{defn}
\label{def:Kurosh_rank}
For a group $G$ and a $G$-tree $T$ with trivial edge stabilisers, the \emph{Kurosh rank} (relative to $T$) of a subgroup $H$ is \[\kappa_{T}(H)=\rank(H\bs T) + |\{Hv \in H\bs T: H_v \neq 1\}| \] where $\rank(H \bs T)$ is the number of edges outside a maximal tree.

The reduced Kurosh rank of a subgroup is $\kr_T(H) = \max\{\rrank{\kappa}_T(H)-1,0\}$.
\end{defn}

Note that $\rank(H \bs T)$ is also the rank of the fundamental group of the graph $H\bs T$. In particular, in the case of a free group acting freely on a tree, this reduces to the usual definition of (reduced) rank. 

\begin{prop}[{\cite[Proposition 2.3]{antolin2014kurosh}}]
\label{prop:subtree_invariance_kr} Let $H$ be a subgroup of $G$, and $T$ a $G$-tree with trivial edge stabilisers. Then
\begin{enumerate}
\item The Kurosh rank with respect to any $H$-invariant subtree $T'$ of $T$, $\kappa_{T'}(H)$ is equal to the Kurosh rank with respect to the minimal $H$-invariant subtree $T_H$, $\kappa_{T_H}(H)$.
\item The Kurosh rank $\kappa_{T}(H)$ is finite if and only if the quotient $H\bs T_H$ is finite.
\end{enumerate}
\end{prop}

If $T$ is clear from context then we will often just write $\kr$, without subscripts, even if we are reasoning with some $H$-subtree.

By Grushko's theorem~\cite{GrushkoDecomp}, the Kurosh rank of a subgroup is bounded above by its true rank. (Each vertex group adds $1$ to the Kurosh rank while adding at least $1$ to the true rank of the subgroup.)

We need to calculate the index of a subgroup from its covering graph of groups, for which we use the following lemma:

\begin{lem}
\label{lem:counting_orbits}
Suppose a group $G$ acts transitively on a set $X$, and $H$ is a subgroup of $G$. Let $X_0$ be a set of orbit representatives for the action of $H$ on $X$. Then \[[G:H] = \sum_{x \in X_0} [G_x:H_x].\] In particular, $[G:H]$ is finite if and only if $X_0$ is finite as is every index $[G_x:H_x]$.
\end{lem}

\begin{proof}
We will exhibit a bijection (though it is very far from canonical) from $\coprod_{x \in X_0} H_x \bs G_x$ to $H \bs G$. To define this, fix a base point $x_0 \in X$, 
and for every element $x$ of $X$ a $g[x] \in G$ such that $g[x]x_0=x$.

Now define a function sending a coset $H_xg$ to the coset $Hgg[x]$. This is well defined in the sense that if $g_1$ and $g_2$ are in the same $H_x$-coset of $G_x$, then $(g_1g[x])(g_2g[x])^{-1}=g_1g_2^{-1}$ which is an element of $H_x$ and in particular $H$. To see it is injective, suppose there are orbit representatives $x,y$ and elements $g_1 \in G_x, g_2 \in G_y$ such that $Hg_1g[x]=Hg_2g[y]$. That is, $g_1g[x]g[y]^{-1}g_2^{-1}$ is an element of $H$. But this element moves $y$ to $x$, and so these must represent the same orbit. So we may assume $x=y$, and this reduces to considering $g_1g_2^{-1}$. This is an element of $H$, and also of $G_x$, and therefore of $H_x$, so $g_1$ and $g_2$ represent the same $H_x$-coset.

To see surjectivity, we will write an arbitrary $g \in G$ as a product $h\hat{g}g[x]$, where $\hat{g} \in G_x$. To do this, let $y$ be the element $gx_0$, and observe that we may write $y$ as $hx=hg[x]x_0$ for some orbit representative $x$. The element $hg[x]$ is in the same $G_y$-coset as $g$, so we may write $g=\tilde{g}hg[x]$ with $\tilde{g}\in G_y$. But recall that $G_y=hG_xh^{-1}$, so we have that $\tilde{g}=h\hat{g}h^{-1}$, with $\hat{g}$ an element of $G_x$. In particular, we have that $g=(h\hat{g}h^{-1})hg[x]=h\hat{g}g[x]$, as required.
\end{proof}

Given any (not necessarily transitive) action, we can calculate the index of a subgroup by restricting our attention to one orbit and using Lemma~\ref{lem:counting_orbits}. In the context of an action on a tree, this would mean looking at the orbit of a single edge or vertex and considering the edge or vertex groups that arise in the covering graph of groups. In particular, for a free product, counting the occurences of any edge in the same $G$-orbit will give the the index. If the original action gave rise to a finite quotient graph, so will the action of a finite index subgroup. In particular, if the original action was minimal, the minimal invariant subtree for the subgroup will be the whole tree again.

This gives a criterion for a subgroup of a graph of groups to be finite index: this happens if and only if the covering graph of groups has finite underlying graph, and every vertex (or indeed edge) group is finite index in the relevant vertex group of the original graph of groups.

Note that this lemma provides a link between the index and Kurosh rank of a finite index subgroup of a free product, though this is complicated somewhat by the possibility that a vertex has non-trivial stabiliser under $G$ but is trivially stabilised by the subgroup $H$. However, if we reduce to the case that the vertex groups are infinite these complications disappear. (Restricting to normal subgroups provides a different simplification.)

\begin{cor}
\label{cor:index_formula_big}
Suppose a group $G$ is expressed as a free product of infinite groups, and $H$ is a finite index subgroup of $G$. Then \hbox{$\kr(H)=[G:H]\kr(G)$}.
\end{cor}
\begin{proof}
Represent $G$ as the fundamental group of a graph of groups $\G$ with trivial edge stabilisers, and let $\mathcal{H}$ be the covering graph of groups corresponding to the action of $H$ on the Bass--Serre tree of $\G$. Write $\kr(H)=|E\mathcal{H}|-|\{v\in V\mathcal{H}:H_v=1\}|$. If a vertex has trivial stabiliser under $H$, it will also have trivial stabiliser under $G$, since otherwise $G_v$ would be a finite subgroup. So use Lemma~\ref{lem:counting_orbits} to rewrite as follows: \begin{align*}
\kr(H)&=|E\mathcal{H}|-|\{v\in V\mathcal{H}:H_v=1\}| \\
&= [G:H]|E\mathcal{G}|-[G:H]|\{v\in V\mathcal{G}:G_v=1\}| \\
&=[G:H]\kr(G). \qedhere
\end{align*}
\end{proof}

For a free group acting either freely or as a free product of infinite cyclic groups, the Kurosh rank and the true rank agree, so this recovers Schreier's formula for free groups: \[\rank(H)-1=[F:H](\rank(F)-1).\]

\section{Subgroup separability for free products}

In this section we provide the Bass--Serre theoretic proof of the theorem that free products of subgroup separable groups are themselves subgroup separable. The proof is in three steps, dealing in turn with ``completing'' a graph of groups immersion, enlarging the vertex groups, and then doing this in general for all finitely generated subgroups of a free product.

Stallings' proof (in the free group case) begins with a (finite) labelled graph where the labelling provides an immersion to a rose. The lifts of any edge in a cover would provide a bijection from the vertex set to itself, and the lifts present in the immersion give a partial bijection. Thus any way of completing the partial bijection to a full bijection is admissible in a cover. (There are only finitely many options, and all of them will work.) The condition on a cover can be checked one edge of the rose at a time, so we may do this separately to each edge and the final graph will be a covering graph.

There are several obstacles in extending this to graphs of groups. First, the graph (of groups) we are covering may have more than one vertex. In the free group case this is easily surmountable, although a little care is needed: we must first make sure that all vertices have equal numbers of lifts. We can achieve this by adding isolated vertices to the immersed graph, to take each vertex up to the maximum. Notice that after adding edges (which now correspond to bijections between different kinds of vertex, in general), every vertex will be connected to at least one vertex of every kind. Since at least one vertex had no extra preimages added, all of these are connected. Thus the constructed graph will have only one component.

Another obstacle is that the notion of ``local injectivity'' is different, and we will need to assign $\delta_e$ values to any new edges in a way which preserves the immersion. Finally, we will have to alter the vertex groups so that all of them are finite index, otherwise (by Lemma~\ref{lem:counting_orbits}) the subgroup cannot be.

We deal in turn with the obstacles presented in generalising Stallings' proof. First, Theorem~\ref{thm:embiggening_the_graph} gives a way to complete an immersion to a cover when the vertex groups have finite index image in the target graph of groups; the Theorem~\ref{thm:embiggening_the_vertex_groups} gives sufficient conditions for enlarging the vertex groups so this condition is met. Finally we prove the Burns--Romanovskii theorem by combining these to produce a covering graph of groups containing a given finite index subgroup but excluding any given element outside it.

\begin{thm}
\label{thm:embiggening_the_graph}
Suppose $G$ is a free product, expressed as the fundamental group of a graph of groups $\mathcal{G}$ where every edge group is trivial. Suppose $H$ is a subgroup of $G$, corresponding to an immersion $\Phi: \mathcal{H} \to \mathcal{G}$, where $\Gamma_{\mathcal{H}}$ is finite and each $H_v$ is mapped to a finite index subgroup of $G_{\varphi(v)}$.

Then there is a finite index subgroup $M$ of $G$ containing $H$ as a free factor.
\end{thm}

\begin{proof}
By Lemma~\ref{lem:counting_orbits}, in a cover the index of the subgroup can be calculated by looking at the preimages of any edge or vertex and their stabilisers. So we need to ensure these are equal. To this end, for each vertex $u$ of $\mathcal{G}$ calculate 
\[ d_u= \sum_{v \in \phi^{-1}(u)} [G_u:H_v].\]
Each $d_u$ is finite: the sum is over finitely many vertices (since $\mathcal{H}$ is finite), and each $[G_u:H_v]$ is finite by assumption. Since $\mathcal{G}$ is also finite, there is a maximum among the $d_u$, say $d$. This will be the degree of the cover. For each vertex in $\mathcal{G}$ add $d-d_u$ isolated vertices to $\mathcal{H}$, declaring them to be in the pre-image of $u$, and assigning each the full subgroup $G_u$. (Recalculating $d_u$ after doing this, all are equal to $d$.)

Though it is disconnected, we can still extend $\Phi$ to the new vertices: each $v$ is in the pre-image of some vertex $u$ of $\mathcal{G}$, and set each new $\phi_v$ to be the identity map. We now need to add edges, further extending the morphism $\Phi$ by assigning $\varphi(e)$ and $\delta_e$ as we do.

We have the local maps on cosets, $\Phi_{v/f}$ and we may extend these to the new vertices (note that, where there are no edges in the pre-image of $f$ at $v$, this is a map from the empty set). These maps are all injections since we began with an immersion and maps out of the empty set must be injective. Our goal is that they should all be bijections: we will need to add more edges, choosing values for $\delta_e$ to achieve this.

For each edge $f$ of $\mathcal{G}$, there should be $d$ pre-images in $\mathcal{H}$. Consider a vertex $v$ of $\mathcal{H}$, and the pre-images $e$ of $f$ at $v$. The values $\delta_e$ form a partial system of coset representatives for $G_{\varphi(v)}/\phi_v(H_v)$, since $\Phi_{v/f}$ is injective.

Suppose $f$ has initial vertex $u$ and terminal vertex $y$. The edges in the pre-image of $f$ provide a partial bijection \[\coprod_{v \in \varphi^{-1}(u)} G_{\varphi(v)}/\phi_v(H_v) \to \coprod_{x \in \varphi^{-1}(y)} G_{\varphi(x)}/\phi_x(H_x),\] by \[\phi_{\iota(e)}(H_{\iota(e)}) \delta_e \mapsto \phi_{\tau(e)}(H_{\tau(e)})\delta_{\overline{e}}.\] 

Both these disjoint unions have cardinality $d$, so this can be completed to a bijection. Add new edges (in the pre-image of $f$) and coset representatives $\delta_e$ and $\delta_{\overline{e}}$ according to this bijection. (The choice of bijection will usually change the subgroup we construct, but not its index.)

Let $\mathcal{M}$ be the graph of groups constructed by repeating this for each edge in $\mathcal{G}$, and $\Phi$ be the extension of the original morphism to all of $\mathcal{M}$. This process added finitely many edges to $\mathcal{H}$. Every connected component of $\mathcal{M}$ contains at least one pre-image of each vertex of $\mathcal{G}$, and since at least one vertex of $\mathcal{G}$ had no pre-images added, and this means the underlying graph of $\mathcal{M}$ will be connected. Also, each $\Phi_{v/f}$ is now bijective, so the morphism $\Phi$ has been extended to a cover.

Picking a base point for $\mathcal{M}$ (in the pre-image of a chosen base point for $\mathcal{G}$) we recover a subgroup $M$ of $G$, which has index $d$ since $\mathcal{M}$ by Lemma~\ref{lem:counting_orbits}.
\end{proof}

Just as in the free group case, restricting $\mathcal{M}$ to the edges and vertices of $\mathcal{H}$ recovers $\mathcal{H}$: so we may view $H$ as a free factor of $M= \pi_1(\mathcal{M})$.

For a general subgroup $H$ of $G$, the vertex groups of $\mathcal{H}$ are not finite index subgroups of the corresponding vertex groups of $\mathcal{G}$ so the process used proving Theorem~\ref{thm:embiggening_the_graph} will not terminate -- in fact, any cover must have infinite degree by Lemma~\ref{lem:counting_orbits}, so there will be infinitely many edges in each pre-image.

So to say anything for general $H$, we must first replace each vertex group $H_v$ with a group mapping to a finite index subgroup of $G_{\varphi(v)}$. Done carelessly, this enlarging of vertex groups is likely to cause some $\delta_e$ values to represent the same coset, and we will no longer have an immersion. So care - and separability assumptions - will be needed as we do this.

\begin{thm}
\label{thm:embiggening_the_vertex_groups}
Suppose $G$ is a free product, expressed as the fundamental group of a graph of groups $\mathcal{G}$ where every edge group is trivial. Suppose $H$ is a subgroup of $G$, corresponding to an immersion $\Phi: \mathcal{H} \to \mathcal{G}$ with $\Gamma_H$ finite. If each $\phi_v(H_v)$ is separable in $G_{\varphi(v)}$, then there is a finite index subgroup $K$ of $G$ corresponding to a cover $\mathcal{K}$ that contains $\mathcal{H}$ as a subgraph of subgroups.
\end{thm}

\begin{proof}
Our first goal is to alter $\mathcal{H}$ and $\Phi$, so each vertex group maps to a finite index subgroup of the relevant $G_v$, while keeping $\Phi$ an immersion. In order to achieve this, we must ensure that the elements $\delta_e$ continue to represent different cosets $G_{\varphi(v)}/\phi_v(H_v)$.

For each vertex $v$ of $\mathcal{H}$ and edge $f$ with $\iota(f)=\varphi(v)$, let $X_{v/f}$ be the finite set of elements $\delta_{e_i}^{-1}\delta_{e_j}$ where $e_i$ and $e_j$ are distinct edges with $\iota(e_i)=\iota(e_j)=v$ and $\varphi(e_i)=\varphi(e_j) = f$. Let $X_v$ be the disjoint union of the $X_{v/f}$ over edges $f$ at $\varphi(v)$.

Since each $\phi_v(H_v)$ was assumed separable in $G_{\varphi(v)}$, there is a finite index subgroup of $G_{\varphi(v)}$ that contains $\phi_v(H_v)$ but no elements of $X_v$. Let $K_v$ be isomorphic to this subgroup, contain $H_v$, and extend $\phi_v$ to $K_v$ so its image is this subgroup.

This remains an immersion: the vertex maps are still injective, so it remains to check the local coset maps $\Phi_{v/f}$. This amounts to checking that for each edge $f$ at $\varphi(v)$ the elements $\delta_e$ with $e$ in the preimage of $f$ represent different right cosets of $\phi_v(K_v)$, or equivalently that $\delta_{e_i}^{-1}\delta_{e_j}$ are outside $\phi_v(K_v)$ for all pairs $e_i,e_j$ of edges.

But this is exactly what we ensured by requiring $\phi_v(K_v)$ to exclude $X_v$, so this condition is satisfied, and $\Phi$ remains an immersion.

Now we are able to apply Theorem~\ref{thm:embiggening_the_graph}: we have an immersion where the vertex groups correspond to finite index subgroups. This immersion can be completed to a cover corresponding to a finite index subgroup. Since the procedure to do this does not identify any edges or vertices, we can recover $\mathcal{H}$ (and the original immersion) by restricting to a subgraph and subgroups of the vertex groups.
\end{proof}

\begin{rem}
Note that the hypothesis that each vertex group embedding is separable was stronger than needed for the conclusion: all that was used was the fact that each $\phi_v(H_v)$ was contained in a finite index subgroup excluding $X_v$ in order to keep the cosets separate. This condition is necessary as well as sufficient, since otherwise there is no way to enlarge $H_v$ to a finite index subgroup where the $\delta_{e_i}$ represent different cosets.
\end{rem}

Now we have the tools to prove the main theorem,

\begin{main_thm*}
Suppose $G$ is a finite free product of subgroup separable groups. Then $G$ is itself subgroup separable.
\end{main_thm*}

\begin{proof}
The proof is essentially an application of Theorem~\ref{thm:embiggening_the_vertex_groups} (and therefore also of Theorem~\ref{thm:embiggening_the_graph}), but a little care is needed in the set up to be sure that we can exclude any element.

Let $\mathcal{G}$ be a finite graph of groups (with trivial edge groups) representing $G$ as $\pi_1(\mathcal{G},v_0)$, and $T$ be the Bass--Serre tree for $\mathcal{G}$. Let $v$ be the vertex in $T$ represented by $G_{v_0}$ - the ``preferred lift'' of $v_0$ to $T$.

Let $H$ be any finitely generated subgroup of $G$ and let $g$ be any element of $G$ outside of $H$. Define a subtree $T_1$ of $T$ as the smallest subtree which \begin{itemize}
\item is $H$-invariant (that is, contains $T_H$);
\item contains $v$;
\item contains $vg$.
\end{itemize}

This subtree is the union of $T_H$ and the orbits of paths from $v$ and $vg$ to $T_H$; since these paths are finite the quotient $T_1/H$ is again finite.

Let $u_0$ be the vertex of $\mathcal{H}$ corresponding to the $H$-orbit of $v$: this will be our basepoint for $\mathcal{H}$. Let $\mathcal{H}$ be the quotient graph of groups obtained from the action of $H$ on $T_1$, choosing the subtree $S^v$ to contain $v$.

The group and graph inclusions $H \to G$ and $T_1 \to T$ induce an immersion $\Phi: \mathcal{H} \to \mathcal{G}$, as described in Proposition~\ref{prop:induced_morphism}. This realises the embedding of $H$ into $G$ as $\Phi_{u_0}: \pi_1(\mathcal{H},u_0) \to \pi_1(\mathcal{G},v_0)$; careful choice of subtrees avoids any conjuations appearing in this map, so a cyclically reduced loop at $u_0$ will be mapped to a cyclically reduced loop at $v_0$

Since $H$ is finitely generated, it also has finite Kurosh rank since this is bounded above by the true rank. So Proposition~\ref{prop:subtree_invariance_kr} implies that $\mathcal{H}$ is a finite graph of groups, since in constructing it we used a tree which differed from $T_H$ only by the addition of the orbits of two finite paths. Notice that since the Kurosh rank must not change when we add these paths, there are no non-trivial vertex stabilisers: the ``interest'' will be captured by the elements $\delta_e$ of the immersion $\Phi$.

The path $p$ from $vg$ to $v$ 
in $T$ and $T_1$ is the lift of $g$ as a loop in $\mathcal{G}$, and quotients to a path beginning at $u_0$ in $\mathcal{H}$. Since $g$ is not an element of $H$, either this path is not a loop, or it is a loop but the final group element $s$ at $G_{v_0}$ is not the image of an element in $H_{u_0}$.

In the first case, we can apply Theorem~\ref{thm:embiggening_the_vertex_groups} directly. Since $G$ was a free product of subgroup separable groups, all the $\phi_v(H_v)$ are separable in $G_{\varphi(v)}$. Since no identifications are made between edges or vertices in this process, the path $p$ will still not be a loop in $\mathcal{K}$, and therefore $\pi_1(\mathcal{K},u_0)$ will not be an element of $K$.

If, on the other hand, $p$ is a loop, we must slightly adapt the proof. When the sets $X_v$ are constructed (containing the elements that must be excluded from each $K_v$), we simply add $s$ to $X_{u_0}$. Then $s$ will still not be the image of an element of $K_{u_0}$, so $g$ is not an element of $K$. \qedhere

\end{proof}

\end{document}